\documentclass{amsart}
\usepackage{mathptmx}

\usepackage{amsmath,amssymb,enumerate}
\usepackage{amsrefs}

\usepackage{hyperref}
\usepackage[all]{xy}

\usepackage{float}

\newtheorem{theorem}{theorem}[section] 
\newtheorem{corollary}[theorem]{Corollary} 
\newtheorem{lemma}[theorem]{Lemma} 
\newtheorem{proposition}[theorem]{Proposition} 

\theoremstyle{definition}

\newtheorem{remark}[theorem]{Remark}  
\newtheorem{question}[theorem]{Question}  
\newtheorem*{ack}{Acknowledgements}

\numberwithin{equation}{section}  
  
 \DeclareMathOperator{\edim}{edim}

\DeclareMathOperator{\card}{card}

\DeclareMathOperator{\Ker}{Ker}

\DeclareMathOperator{\Ass}{Ass}

 \DeclareMathOperator{\Ext}{Ext}
 \DeclareMathOperator{\Tor}{Tor}
 
 \DeclareMathOperator{\pdim}{pdim}
\DeclareMathOperator{\depth}{depth}
 
\DeclareMathOperator{\reg}{reg}

\DeclareMathOperator{\rank}{rank}

\newcommand{\fp}{\frak{p}}

\newcommand{\bbZ}{\mathbb{Z}}

\newcommand{\ee}{\varepsilon}

\newcommand{\bsf}{{\boldsymbol f}}
\newcommand{\bsz}{{\boldsymbol z}}

\newcommand{\lar}{\longrightarrow}
\newcommand{\xra}{\xrightarrow}

\newcommand{\ges}{\geqslant}
\newcommand{\les}{\leqslant}

\title{Koszul Algebras Defined by Three Relations}
\author[A.Boocher]{Adam Boocher}
\address{University of Utah, Salt Lake City, Utah, USA}
\email{boocher@math.utah.edu and  aboocher@gmail.com}

\author[S. H.~Hassanzadeh]{S. Hamid Hassanzadeh}
\address{Instituto de Matem\'{a}tica,
Universidade Federal do Rio de Janeiro,  Brazil. }
\email{hamid@im.ufrj.br and  hassanzadeh.ufrj@gmail.com}

\author[S. B. Iyengar]{Srikanth B.  Iyengar}
\address{University of Utah, Salt Lake City, Utah, USA}
\email{iyengar@math.utah.edu}

\begin{document}
\begin{abstract}
This work concerns commutative algebras of the form $R=Q/I$, where $Q$ is a standard graded polynomial ring and $I$ is a homogenous ideal in $Q$.  It has been proposed that when $R$ is Koszul the $i$th Betti number of $R$ over $Q$ is at most $\binom gi$, where $g$ is the number of generators of $I$; in particular, the projective dimension of $R$ over $Q$ is at most $g$. The main result of this work settles this question, in the affirmative, when $g\le 3$. 
\end{abstract}


\date{2nd May 2017}

\maketitle
\maketitle

%
%

\begin{center}
\emph{To Professor Winfried Bruns, on the occasion of his 70th birthday.}
\end{center}

\section{Introduction}
This work is about the homological properties of homogeneous affine algebras, that is to say, algebras $R$ of the form $Q/I$ where $Q=k[x_{1},\dots,x_{e}]$, with each $x_{i}$ of degree one, and $I$ is a homogenous ideal in $Q$. The emphasis is on \emph{Koszul} algebras: algebras $R$ with the property that  $\Tor^{R}_{i}(k,k)_{j}=0$ whenever $i\ne  j$; equivalently, the minimal graded free resolution of $k$ over $R$ is linear. We are interested in the connection between the Koszul property of $R$ and invariants of $R$ as a $Q$-module; in particular, the graded Betti numbers, namely, the numbers $\beta_{i,j}^{Q}(R):=\rank_{k}\Tor^{Q}_{i}(R,k)_{j}$, and the total Betti numbers $\beta^{Q}_{i}(R):=\sum_{j}\beta^{Q}_{i,j}(R)$.

It has long been known that there is such connection:  Backelin~\cite{BThesis} and Kempf~\cite{Ke} proved that when $R$ is Koszul $\beta_{i,j}^{Q}(R)=0$ when $j>2i$. Said otherwise, in the Betti table of $R$ viewed as an $Q$-module, the nonzero entries all lie on or above the diagonal line.  In \cites{ACI, ACI1} these results have been refined to obtain more stringent constraints on the shape of this Betti table. 
Our focus is on the Betti numbers themselves and in particular the following intriguing question formulated in ~\cite{ACI}.

\begin{question}
\label{qu:conca}
Set $g:=\beta^{Q}_{1}(R)$, the number of relations defining $R$. If $R$ is Koszul algebra, are there inequalities
\[
\beta^Q_i(R) \leq {\binom gi} \quad \text{for each $i$?} 
\]
In particular, is the projective dimension of $R$ over $Q$ is at most $g$? 
\end{question}

This conjecture holds when the ideal $I$ of relations can be generated by monomials, for then the Taylor resolution of $R$ over $Q$ furnishes the desired bounds. It follows by standard arguments that the same is true when $I$ has a Gr\"obner basis of quadrics; or even it has one after a suitable deformation; see Remark~\ref{re:Taylor}. As far as we are aware, no other families of Koszul algebras are known that satisfy this conjecture. On the other hand, it would be remiss of us not to mention that the only known example of a Koszul algebra that does not have a Gr\"obner basis of quadrics is  \cite{Co}*{Example~1.20}.

The main result of this work is that Question~\ref{qu:conca} has an affirmative answer when $g \leq 3$; in particular, Koszul algebras defined by 3 equations have projective dimension at most $3$; see Corollary~\ref{co:conca}. Combining this with earlier work of Backelin and Fr\"oberg~\cite{BF} and Conca~\cite{C2000}, we arrive at a similar conclusion when the embedding dimension of $R$ is at most $3$; see Remark~\ref{re:conca}.

This may seem scant evidence for an affirmative answer to Question~\ref{qu:conca} but there exist ideals  $I$ generated by three quadrics with  $\pdim_{Q}(Q/I)=4$, and that is the largest it can be, by a result of Eisenbud and Huneke; see~\cite{MS}. So the Koszul property is reflected already in this special case. Another reason the preceding result is not without interest is that, by a result of Bruns~\cite{Br2},  essentially every free resolution over $Q$ is  the free resolution of an ideal that can be generated by three elements; these will not, in general, be quadrics. 

Our proof of Corollary~\ref{co:conca} is based on some general results concerning the \emph{multiplicative structure} of $\Tor^{Q}_{*}(R,k)$, which is nothing but the Koszul homology algebra of $R$. We prove that when $R$ is Koszul, the diagonal $k$-subalgebra $\oplus_{i} \Tor^{Q}_{i}(R,k)_{2i}$ is the quotient of the exterior algebra on $\Tor^{Q}_{1}(R,k)_{2}$, modulo quadratic relations that depend only on the first syzygies of $I$; see Theorem~\ref{th:delta} and Remark~\ref{re:delta}. One consequence of this is that, for Koszul algebra, one has $\beta^{Q}_{i,2i}(R)\le \binom gi$ for all $i$. Moreover,  if equality holds for \emph{some} $2\le i\le g$ then $R$ is a complete intersection; the proof of this latter also uses a characterization of complete intersections in terms of the product in the Koszul homology algebra, due to Bruns~\cite{Br2}. These results are proved in Section~\ref{se:koszul-homology}. The arguments exploit the machinery of minimal models of algebras, developed by Avramov~\cite{Av1}. The relevant details are recalled in Section~\ref{se:models}, where they are also used to establish results on almost complete intersection rings, which also play a key role in the proof of our main result.

\begin{center}
  $\ast \ \ \ast \ \ \ast$
\end{center}
\noindent
Throughout the paper $k$ will be a field and $R:=\{R_{i}\}_{i\ges 0}$ a standard graded finitely generated $k$-algebra; in other words, $R$ is generated as a $k$-algebra by $R_{1}$ and  $\rank_{k}R_{1}$ is finite. Let $Q$ be the symmetric $k$-algebra on $R_{1}$ and $Q\to R$ the canonical projection. In particular, $Q$ is a standard graded finitely generated polynomial ring and the ideal $I:=\Ker(Q\to R)$ is homogenous and contained in $Q_{\ges 2}$.

\section{Betti numbers and deviations}
\label{se:models}
In this section we recollect the construction of the minimal model of $R$,  and certain numerical invariants based on them; namely the deviations of $R$ and the Betti numbers of $R$ over $Q$. There is little new, except Theorem~\ref{th:aci}. 

Since $Q\to R$ is a morphism of graded rings, resolutions of $R$ over $Q$, and invariants based on them, have  an \emph{internal} degree induced by the grading on $Q$, in addition to the usual  homological degree. In what follows, given an element $a$ in such a bigraded object we write $\deg(a)$ for the internal degree and $|a|$ for the homological degree; the tacit assumption is that only homogeneous elements of graded objects are considered.

\subsubsection*{Minimal models} 
Let $Q[X]$ be a \emph{minimal model} for $R$ over $Q$. Thus, $X$ is a bigraded set such that for each $n$, the set $X_{n}:=\{x\in X\mid |x|=n\}$ is finite and the graded algebra underlying $Q[X]$ is $Q\otimes_{k}\bigotimes_{n=1}^{\infty}k[X_{n}]$ where $k[X_{n}]$ is the symmetric algebra on the graded $K$-vector space $kX_{n}$ when $n$ is even, and the exterior algebra on that space when $n$ is odd. 
In particular, $Q[X]$ is strictly graded-commutative with respect to homological degree: For $a,b$ in $Q[X]$ one has
\[
ab = (-1)^{|a||b|}ba \quad \text{and}\quad a^{2} = 0 \quad\text{if $|a|$ is odd}.
\]
The differential on $Q[X]$ satisfies the Leibniz rule and is \emph{decomposable}, in that
\begin{equation}
\label{eq:diff-model}
d(x)\subseteq (Q_{\ges 1}+(X))^{2}Q[X] \quad\text{for all $x\in X$.}
\end{equation}
Thus, $Q[X]$ is a DG(=Differential Graded) $Q$-algebra. There is a morphism of  $Q$-algebras $Q[X]\to R$ that is a quasiisomorphism, so that  $Q[X]$ is a DG  algebra resolution of $R$. For details of the construction see \cite{Av1}*{\S7.2}. The first steps can be described explicitly.

\begin{remark}
\label{rem:model-steps}
Let $\bsf:=f_{1},\ldots,f_{g}$ be a minimal generating set for the ideal $I$ and let $Q[X_{1}]$ be the Koszul complex on  $\bsf$. 
Thus, $X_{1}:=\{x_{1},\dots,x_{g}\}$ is a set of indeterminates with $|x_{i}|=1$ and $\deg(x_{i})=\deg(f_{i})$ for each $i$, and $Q[X_{1}]$ is the exterior algebra $QX_{1}$. The differential on $Q[X_{1}]$ is determined by $d(x_{i})=f_{i}$. The canonical augmentation $Q[X_{1}]\to R$ induces an isomorphism $H_{0}(Q[X_{1}]) \cong R$.

Let $\bsz:=z_{1},\dots,z_{l}$ be cycles in $Q[X_{1}]$ that form a minimal generating set for the $R$-module for $H_{1}(Q[X_{1}])$. The next step in the construction of the minimal model is to kill these cycles. In detail: Let $X_{2}:=\{y_{1},\dots,y_{l}\}$ be a graded set with $|y_{i}|=2$ and $\deg(y_{i})=\deg(z_{i})$ for each $i$. With $k[X_{2}]$ the symmetric algebra on $kX_{1}$, one has $Q[X_{\les 2}]=Q[X_{1}]\otimes_{k}k[X_{2}]$. The differential on $Q[X_{\les 2}]$ extends the one on $Q[X_{1}]$ and is determined by $d(y_{i})=z_{i}$. Thus  $Q[X_{1}]$ is a DG subalgebra of $Q[X_{\les 2}]$, the augmentation $Q[X_{1}]\to R$ extends to $Q[X_{\les 2}]\to R$, and satisfies
\[
H_{0}(Q[X_{\les 2}])\cong R \quad\text{and}\quad H_{1}(Q[X_{\les 2}])=0\,.
\]
The next step of the construction is to kill the cycles in $H_{2}(Q[X_{\les}2])$, and so on. One thus gets a tower of DG $Q$-algebras $Q[X_{1}]\subseteq Q[X_{\les 2}]\subseteq \cdots $ whose union is $Q[X]$. By construction, for each $n\ge 1$ one has
\[
H_{i}(Q[X_{\les n}]) = 
\begin{cases}
R & \text{for $i=0$} \\
0 & \text{for $1\leq i \leq n-1$}.
\end{cases}
\]
Moreover, $d(X_{n+1})$ is a minimal generating set for the $R$-module $H_{n}(Q[X_{\les n}])$.
\end{remark}

Henceforth, we fix a minimal model $Q[X]$ for $R$ and set $k[X]:=k\otimes_{Q}Q[X]$. 

\begin{remark}
\label{re:betti-deviations}
It follows from \eqref{eq:diff-model} that the differential on $k[X]$ satisfies $d(X)\subseteq (X)^{2}$. In particular, 
$d(X_{1})=0=d(X_{2})$, so that $H_{1}(k[X])\cong kX_{1}$ as bigraded $k$-vector spaces, and the $k$-vector subspace of the cycles in $k[X]_{2}$ is $kX_{2}\oplus \wedge^{2}kX_{1}$. Moreover, one has
\begin{equation}
\label{eq:betti-deviations}
H_{2}(k[X]) \cong kX_{2}\oplus (\wedge^{2}kX_{1}/d(kX_{3})) \quad\text{and}\quad 
\Ker(kX_{3}\xra{d}\wedge^{2}kX_{1})\subseteq H_{3}(k[X])
\end{equation}
as bigraded $k$-vector spaces.
\end{remark}

Next we recall some numerical invariants that can be read off the minimal model.

\subsubsection*{Betti numbers and deviations}
 Since $Q[X]$ is a DG algebra resolution of $R$ over $Q$, there is an isomorphism 
\[
H(k[X]) \cong \Tor^{Q}(k,R)
\]
of bigraded $k$-algebras. Thus the graded Betti numbers of $R$ over $Q$ are given by
\[
\beta_{ij}^{Q}(R):= \rank_{k}H_{i}(k[X])_{j}\quad \text{and}\quad \beta_{i}^{Q}(R):=\sum_{j} \beta_{ij}^{Q}(R) \,.
\]
The notation notwithstanding, these Betti numbers are invariants of $R$ alone, for they correspond to ranks of Koszul homology modules of $R$; see Remark~\ref{re:Koszul}. The same is the case with the \emph{deviations}, $\{\ee_{ij}(R)\}$, of $R$ which are the integers
\[
\ee_{ij}(R):=\card(X_{i-1,j}) \quad \text{and}\quad \ee_{i}(R):= \sum_{j} \ee_{ij}(R) \quad\text{for $i\ge 1$}\,.
\]
Typically, these invariants are derived from the generating series for  $\Tor^{R}_{i}(k,k)_{j}$; the definition above is justified by  \cite{Av1}*{Theorem 7.2.6}.  

\begin{remark}
The deviations can be estimated in  terms of the homology of the DG subalgebras $Q[X_{\les i}]$ of $Q[X]$ generated by the graded set $X_{\les i}$, for various $i$. Indeed, the exact sequence of complexes
\[
0\lar Q[X_{\les i}]  \lar Q[X_{\les i+1}] \lar Q[X_{\les i+1}] /Q[X_{\les i}] \lar 0
\]
yields, in homology, the exact sequence of graded $R$-modules
\[
H_{i+1}(Q[X_{\les i+1}]) \lar H_{i+1}(Q[X_{\les i+1}] /Q[X_{\les i}])\lar H_i( Q[X_{\les i}])\lar 0\,.
\]
By construction, the term in the middle is $RX_{i+1}$ and the map $RX_{i+1}\to  H_i( Q[X_{\les i}])$ is a minimal presentation. Thus, the exact sequence above yields a  presentation:
\begin{equation}
\label{eq:hi-presentation}
RX_{i+2}\xra{d} RX_{i+1}\lar  H_i( Q[X_{\les i}])\lar 0\,.
\end{equation}
where $d(X_{i+2})\subseteq R_{\ges 1}X_{i+1}$. This discussion justifies the following result.

\begin{lemma}
\label{le:deviations-homology}  
For all $i\geq 1$ and $j\in\bbZ$, there are (in)equalities
\[
\ee_{i+2,j}(R) = \beta_{0,j}^R(H_i(Q[ X_{\les i}]))\quad\text{and}\quad \ee_{i+3,j}(R)\geq \beta_{1,j}^R(H_i(Q[ X_{\les i}]))\,. \qed
\]
\end{lemma} 
\end{remark}

The next result explains why deviations have a bearing on Question~\ref{qu:conca}. 

\begin{proposition}
\label{pr:betti-deviations} 
There are inequalities
\[
-\ee_{3}(R)\le {\binom{\beta_1^Q(R)} 2}-\beta_2^Q(R)\le \ee_4(R)-\ee_3(R);
\]
When $\ee_{ij}(R)=0$ for $i\le 4$ and $j>i$, equality holds on the right iff $\beta_{34}^{Q}(R) =0$.
\end{proposition}

\begin{proof} 
From \eqref{eq:betti-deviations} one gets an equality.
\[
{\binom{\ee_2(R)} 2} - \rank_k H_2(k[X]) =  \rank_k d(kX_3) - \ee_3(R)\,.
\]
The inequalities follow, since $\ee_2(R)= \beta_{1}^{Q}(R)$ and $\rank_k H_i(k[X])=\beta_{2}^Q(R) $.  

The  stated hypothesis on $\ee_{ij}(R)$ implies  $\Ker(kX_{3}\to \wedge^{2}kX_{1}) = H_{3}(k[X])_{4}$. This justifies the last assertion.
\end{proof}

The inequalities in Proposition~\ref{pr:betti-deviations} can be strict.

\begin{remark}
\label{le:residual}
Recall that $I=\Ker(Q\to R)$. Assume that for some prime $\fp\supseteq I$ the ideal $I_{\fp}$ is generated by a regular sequence of length $\beta^{Q}_{1}(R)$. For example, this is so if $R$ is Cohen-Macaulay and a residual intersection; see~\cite{HN}*{5.8}. There are inequalities
\[
\binom{\beta_1^Q(R)} 2 - \beta_2^Q(R) \leq 0 \quad\text{and}\quad 0 \leq  \ee_4(R) -  \ee_3(R) \,.
\]
Indeed, the hypotheses on $I_{\fp}$ yields, for each $i$, the equalities below.
\[
\beta_i^Q(R) \geq \beta^{Q_{\fp}}_{i}(R_{\fp}) = \binom{\beta^{Q_{\fp}}_{1}(R_{\fp})}i = \binom{\beta^{Q}_{1}(R)}i\,.
\]
The inequality is standard. As $H_{1}(Q[X_{1}])_{\fp}=0$, the assertion about the deviations is obtained by localizing the exact sequence \eqref{eq:hi-presentation} for $i=1$, at the prime ideal $\fp$. 

Both inequalities can be strict: For $Q=k[x_1,\ldots,x_4]$ and $I=(x_1^2-x_3^2,x_1x_2,x_3x_4)$, the Betti table of $R$ over $Q$ is 
 \begin{center}
\begin{tabular}{c|clclcl}
{}&0&1&2&3\\
\hline
0&1&--&--&--\\
1&--&3&--&--\\
2&--&--&4&2
\end{tabular}
\end{center}
On the other hand, a direct computation shows that $\ee_3(R)=1$ and $\ee_4(R)=2$. 
\end{remark}

\subsubsection*{Almost Complete Intersections}
\label{se:aci}
The ring $R$ is \emph{almost complete intersection} if it satisfies $\beta^{Q}_{1}(R) = \dim Q - \dim R +1$; in words, if the minimal number of generators for $\Ker(Q\to R)$ is precisely one more than its codimension. In what follows $\omega_{R}$ is the canonical module of $R$, namely, the $R$-module $\Ext_{Q}^{c}(R,Q)$, where $c=\dim Q-\dim R$.

\begin{theorem}
\label{th:aci} 
If $R$ is an almost complete intersection, then $\ee_3(R)\leq \ee_4(R)$.
\end{theorem}

\begin{proof}
Kunz~\cite{Ku}*{Proposition~1.1} proved that the canonical module of an almost complete intersection is not free. In the minimal model $Q[X]$, the DG algebra $Q[X_{1}]$ is the Koszul complex on $I$.  As $R$ is almost complete intersection $H_1(Q[X_1])\cong \omega_{R}$; see~\cite{Ku}*{Proposition~2.1}.  The desired result thus follows from Lemma~\ref{le:deviations-homology} and Lemma~\ref{le:canonical} below.
\end{proof}

\begin{lemma}
\label{le:canonical} 
If  $\beta_0^R(\omega_R)> \beta_1^R(\omega_R)$, then $\omega_R$ is free, of rank one.
\end{lemma}

\begin{proof}  
The grading plays no role in what follows, so is ignored. One has $\omega_R\cong \omega_{Q/J}$, where $J$ is the intersection of the primary component of $I$ with height equal to the height of $I$. It is then easy to check that there are (in)equalities 
\[
\beta_1^{R}(\omega_R)\geq \beta_1^{Q/J}(\omega_R)\quad\text{and}\quad
\beta_0^{R}(\omega_R)=\beta_0^{Q/J}(\omega_R)\,.
\]
Thus, it suffices to verify the result for $Q/J$, and so, replacing $R$ by $Q/J$, we may assume that $I$ has no embedded associated primes. Consider a minimal presentation of $\omega_{R}$.
\[
0\lar Z\lar R^{\beta_{1}}\lar R^{\beta_{0}}\lar \omega_{R}\lar 0\,.
\]

For prime $\fp\in \Ass R$, the $R_{\fp}$-module $(\omega_R)_{\fp}\cong \omega_{R_{\fp}}$ is the canonical module of $R_{\fp}$, which is a local ring of dimension zero.  With $\ell(-)$ denoting length, from the exact sequence above one then gets 
\[
\ell(Z_{\fp})= (\beta_1-\beta_0+1)\ell(R_{\fp}) < 0\,.
\]
The inequality is by the hypothesis. Thus, $Z_{\fp}=0$ for each $\fp\in \Ass R$, which yields $Z=0$. Thus, $\omega_R$ has finite projective dimension.   Aoyama's result \cite{Ao}*{Theorem~3} then implies that $\omega_{R}$ is free of rank one as desired.
\end{proof}

\begin{remark}
Jorgensen  and Leuschke~\cite{JL}*{Question 2.6} ask: Is a Cohen-Macaulay ring $R$ with $\beta_0^R(\omega_R)\geq \beta_1^R(\omega_R)$ Gorenstein? Lemma~\ref{le:canonical} settles it when the inequality is strict. The Cohen-Macaulay assumption is needed; consider $R=k[|x,y|]/(x^2,xy)$.
\end{remark}

\section{The Koszul homology algebra}
\label{se:koszul-homology}
This section concerns the Koszul homology algebra of $R$. The main result, Theorem~\ref{th:delta}, describes its diagonal subalgebra, in the sense explained further below, when the resolution of $k$ over $R$ is linear for the first few steps.

Let $K^{R}$ be the Koszul complex of $R$; see \cite{BH}*{\S1.6}.  By construction, $K^{R}$ is a DG $R$-algebra  whose underlying graded algebra is $R\otimes_{k} \bigwedge_{k }V$, where $V=\Sigma R_{1}$, the $k$-vector space $R_{1}$ in (homological) degree one. The differential on $K^{R}$ is $R$-linear and defined by $d(\Sigma v)=v$  for $v\in R_{1}$ and the Leibniz rule. Observe that $K^{R}$ is bigraded, with internal degree inherited from $R$, and strictly graded-commutative with respect to the homological degree. Its homology algebra, $H(K^{R})$, inherits these properties. The next remark is well known. 

\begin{remark}
\label{re:Koszul}
Let $Q[X]$ be the minimal model for $R$ over $Q$, introduced in Section~\ref{se:models}, and set $k[X]:=k\otimes_{Q}Q[X]$. There is are isomorphisms of bigraded $k$-algebras 
\[
H(K^{R})\cong H(k[X]) \cong  \Tor^{Q}(k,R)\,.
\]
Indeed, the Koszul complex  $K^{Q}$ of $Q$ is a free resolution of $k$ over $Q$, and $Q[X]$ is free resolution of $R$. Hence there are quasiisomorphisms of DG algebras
\[
k\otimes_{Q}Q[X] \xleftarrow{\ \simeq\ } K^{Q}\otimes_{Q} Q[X] \xrightarrow{\ \simeq\ } K^{Q}\otimes_{Q}R \cong K^{R}\,.
\]
In homology, this yields the stated isomorphism between $H(K^{R})$ and $H(k[X])$. The second isomorphism has been commented on earlier.
\end{remark}

It follows from the preceding isomorphisms that $\rank_{k}H_{i}(K^{R})_{j}=\beta^{Q}_{i,j}(R)$. The focus of this section is the $k$-subalgebra
\[
\Delta(R) = \bigoplus_{i\ge 0} H_{i}(K^R)_{2i}
\]
that we call the \emph{diagonal subalgebra} of $H(K^R)$. It is supported on the main diagonal of the Betti table of $R$ over $Q$. Being a subalgebra of $H(K^{R})$, the $k$-algebra $\Delta(R)$ it also strictly graded-commutative, with $\Delta(R)_{1}=\Delta(R)_{1,2}=H_{1}(R)_{2}$. Hence, by the universal property of exterior algebras, there is a morphism of graded $k$-algebras
\begin{equation}
\label{eq:diagonal-map}
\rho\colon \bigwedge_{k} H_{1}(R)_{2} \lar \Delta(R)\,.
\end{equation}

This map is surjective when $R$ is Koszul, and then one has a concrete description of its kernel; see Theorem~\ref{th:delta} below.

\subsubsection*{Koszul algebras and regularity}
Recall that the $k$-algebra $R$ is \emph{Koszul} if $\Tor_{i}^{R}(k,k)_{j}=0$ unless $i=j$; equivalently, if the minimal resolution of $k$ over $R$ is linear. We need a weakening of this condition, and to this end  recall an invariant introduced in \cite{ACI}*{\S4}: The \emph{$n$th partial regularity} of the $R$-module $k$ is the number
\[
\reg^{R}_{n}(k):=\sup\{j - i \mid \text{where $i\le n$ and $\Tor^{R}_{i}(k,k)_{j}\ne 0$}\}.
\]
Thus, $R$ is Koszul precisely when $\reg^{R}_{n}(k)\le 0$ for each $n$. We are particularity interested in the condition $\reg^{R}_{n}(k)= 0$ that translates to the condition that the minimal resolution of $k$ over $R$ is linear up to degree $n$; equivalently, $\ee_{ij}(R)=0$ for $i\le n$ and  $j\ne i$. Given the description of the deviations in terms of the minimal model $Q[X]$ of $R$, described in Section~\ref{se:models}, one gets:
\begin{equation}
\label{eq:koszul-model}
\text{$\reg^{R}_{n+1}(k)= 0$ if and only if $X_{i,j} = \varnothing$ for $i\le n$ and  $j\ne i+1$.}
\end{equation}
This equivalence plays an important role in the sequel. Roos~\cite{Roos} has constructed, for each integer $n\ge 2$, a $k$-algebra $R$ that is \emph{not} Koszul and has $\reg^{R}_{n+1}(k)=0$.

We also repeatedly use the following facts established in \cite{ACI}*{Theorem~4.1}.

\begin{remark}
\label{re:ACI}
When $\reg^{R}_{n+1}(k)=0$ for an integer $n$, then for each $0\le i\le n$ one has:
\begin{enumerate}
\item
$\beta^{Q}_{i,j}(R) =0$ for $j>2i$;
\item
$\Delta(R)_{i}=\Delta_{1}(R)^{i}$.
\end{enumerate}
\end{remark}

The statement, and proof, of the next result is an elaboration of the remark above.

\begin{theorem}
\label{th:delta}
Assume $\reg^{R}_{n+1}(k)=0$ for $n=\edim R- \depth R$. With $Q[X]$ denoting the minimal model of $R$, the differential on $k[X]:=k\otimes_{Q}Q[X]$ satisfies $d(X_{3})\subseteq  k[X_1]_{2,4}$ and there is an isomorphism of $k$-algebras
\[
\Delta(R) \cong  k[X_{1}]/(d(X_{3}) )\,.
\]
\end{theorem}

\begin{proof}
Since $d(X_{1})=0$, as noted in Remark~\ref{re:betti-deviations}, the differential on $k[X_{1}]$ is zero as well. Consider the inclusion $k[X_{1}]\subseteq k[X]$ of DG algebras. The map \eqref{eq:diagonal-map} is realized as the induced map on homology:
\[
\rho\colon  k[X_{1}]\lar \bigoplus_{i\geq 0}H_{i}(k[X])_{2i}\,.
\]
For any monomial $x_{1}^{d_{1}}\cdots x_{s}^{d_{s}}$ in $k[X]$ with $x_{n}\in X$, the hypothesis on regularity, in the form \eqref{eq:koszul-model}, yields
\[
\deg(x_{1}^{d_{1}}\cdots x_{s}^{d_{s}}) = \sum_{n=1}^{s} (|x_{n}| + 1)d_{n}\,.
\]
Given this, an elementary computation yields that $\deg(x_{1}^{d_{1}}\cdots x_{s}^{d_{s}})=2|x_{1}^{d_{1}}\cdots x_{s}^{d_{s}}|$ precisely when $|x_n|=1$ for $n=1,\ldots,s$. Thus, the diagonal subalgebra $\oplus_{i} k[X]_{i,2i}$ of  $k[X]$ is $k[X_{1}]$. Since $H(k[X])$ is subquotient of $k[X]$ and it follows that $\rho$ is surjective. 

To verify the claim about its kernel, it suffices to verify that there is an equality
\[
d(k[X])\cap k[X_{1}] = d(X_{3})k[X_{1}]\,.
\]
Again from Remark~\ref{re:betti-deviations} one gets $d(X_{2})=0$ and $d(X_{3})\subseteq k[X_{1}]$. Thus it suffices to consider the differential of monomials $x_{1}^{d_{1}}\cdots x_{s}^{d_{s}}$ where $|x_{n}|\geq 3$ for some $n$ in $1,\ldots,s$. For such a monomial, it is easy to verify that
\[
\deg(x_{1}^{d_{1}}\cdots x_{s}^{d_{s}})=2(|x_{1}^{d_{1}}\cdots x_{s}^{d_{s}}|-1)
\]
if and only if $|x_{n}|=3$ for exactly one $n$ in $1,\ldots,s$, and then $d_{n}=1$, and $x_{n}=1$ for the rest. This is the desired result.
\end{proof}

The following corollary will be useful in the sequel.

\begin{corollary}
\label{Co:diagonalCI}
Assume $\reg_{n+1}(R)=0$ for $n = \edim R - \depth R$.  If $\beta_{i,2i}^{Q}(R)\ge {\binom{\beta_{1}^{Q}(R)} i}$ for some $ 2\le i\le \beta_{1}^{Q}(R)$, then $R$ is a complete intersection. 
\end{corollary}

\begin{proof}
By Theorem \ref{th:delta} the stated inequality holds if and only if map $d(X_3) = 0$; equivalently, when the map $\rho$ is injective. A theorem of Bruns~\cite{Br2}*{Theorem~2}, see also \cite{BH}*{Theorem 2.3.14}, then yields that $R$ is complete intersection.
\end{proof}

Next we give an elementary description of $d(kX_{3})$ appearing in Theorem~\ref{th:delta}.  This is based on the construction of the model $Q[X]$, described in Remark~\ref{rem:model-steps}.

\begin{remark}
\label{re:delta}
The set up and hypotheses is as in Theorem~\ref{th:delta}. 

Suppose $X_{1}=\{x_{1},\dots,x_{g}\}$ and set $f_{i}=d(x_{i})$. By construction, $\bsf:=f_{1},\ldots,f_{g}$ is a minimal generating set for the ideal $\Ker(Q\to R)$, and $Q[X_{1}]$ is the Koszul complex on  $\bsf$. Let  $l_1, \ldots l_r$ be a  generating set for linear syzygies of $\bsf$; that is to say, for the $k$-vector space of cycles in $Q[X_{1}]_{1,3}$. Set $L=\sum_{i}Q_{1}l_{i}$, the $k$-vector subspace of cycles in $Q[X_{1}]_{1,4}$ that are generated by the linear syzygies. In the same vein, let $M$ be the $k$-vector subspace of $Q[X_{1}]_{1,4}$ spanned by the  syzygies $\{f_{i}x_{j}-f_{j}x_{i}\}$, where $1\leq i < j\leq g$. Thus $V = L \cap M$ is the $k$-vector space generated by the nonminimal Koszul syzygies. Choose a basis  $b_{1},\ldots, b_{p}$ of the $k$-vector space $V$ and write
\[
b_h =  \sum_{1\leqslant i < j\leqslant g} c_{hij} (f_{i}x_{j}-f_{j}x_{i})\,.
\]
It  follows from the construction of the minimal model that the image of $d$ is the ideal of $k[X_{1}]$ generated by the quadratic forms
\[
\sum_{ij} c_{hij} x_i x_j\quad h=1,\ldots,p.
\]
\end{remark}

\begin{remark}
\label{re:thick-diagonal}
By Remark~\ref{re:ACI}, the subalgebra $\Delta(R)$ of $H(K^{R})$ is generated by its linear part. The same is true of the larger subalgebra $\oplus_{j\ge 2i-1} H_{i,j}$; this  is the content of \cite{BDGMS}*{Theorem~3.1}. Its proof follows the lines of the argument for that of Theorem~\ref{th:delta}.
\end{remark}

\section{Algebras defined by three relations}
\label{sec:3relations}
In this section we investigate the Koszul homology of Koszul algebras defined by three relations, answering Question~\ref{qu:conca} along the way for this class of rings. We begin by collecting some observations which are surely well known. In what follows $R$ and $Q$ are usual; recall that $I=\Ker(Q\to R)$.

\begin{remark}
\label{re:Taylor}
Let $J$ be the initial ideal of $I$ with respect to a term ordering. The semicontinuity of flat degenerations yields inequalities
\begin{equation}
\label{eq:semicontinuity}
\beta_{i,j}(R) \leq \beta_{i,j}(Q/J) \quad\text{for all $i,j$.}
\end{equation}
The ideal $J$ is monomial and its Taylor resolution~\cite{Pe2}*{Construction~26.5} yields 
\[
\beta_{i}^{Q}(Q/J)\leq \binom{\beta_{1}^{Q}(Q/J)}i\,.
\]
Combining the preceding inequalities yields an upper bound for the Betti numbers of $R$ over $Q$ in the spirit of Question~\ref{qu:conca}, and settles it in a special case, namely, when $k$-algebra $R$ is \emph{G-quadratic}; that is to say, when for some choice of term order the initial ideal $J$ is generated by quadratics. In that case $\beta_{1}^{Q}(Q/J)=\beta_{1}^{Q}(R)$, so the estimates above settle Question~\ref{qu:conca}.   Since the Betti numbers of $R$ over $Q$ are invariant under change of coordinates, and quotient by a regular sequence of linear forms, Question~\ref{qu:conca} has an affirmative answer  when $R$ is \emph{LG-quadratic}; see the paragraph preceding~\cite{ACI1}*{Question~6.5} for details.
\end{remark}

\begin{proposition}
\label{pr:initial-ideal}
Assume $I$ is generated by quadrics and set $n:=\beta^{Q}_{1,2}(R)$. The following statements hold.
\begin{enumerate}[\quad\rm(1)]
\item
$\beta^{Q}_{i,i+1}(R) \leq {\binom n i}$ for each $i$, and if equality holds for $ i=2$, then $I$ has codimension one and a linear resolution of length $n$;
\item
$\beta_{i,i+1}^{Q}(R) \leq 2$ for some $i\ge 2$ implies $\beta_{i+1,i+2}^{Q}(R) = 0$;
\item
$\beta^{Q}_{2}(R)\leq 2 \binom n2$ when $\reg^{R}_{3}(k)=0$.
\end{enumerate}
When $J$ is the initial ideal of $I$ with respect to a term ordering and $m:=\beta^{Q}_{1,3}(J)$, the following inequalities hold
\begin{enumerate}[\quad\rm(1)]
\setcounter{enumi}{3}
\item
$\beta^{Q}_{2,3}(R) + \beta^{Q}_{2,4}(R) \leq \binom{n + m}2$;
\item
$\beta^{Q}_{2,3}(R) + m\leq \binom{n}2$.
\end{enumerate}
\end{proposition}

\begin{proof}
For ease of notation we write $\beta(-)$ for $\beta^{Q}(-)$. 

(2): Set $b:= \beta_{i-1,i}(R)$ and  $c:= \beta_{i,i+1}(R)$. Choosing bases, the linear strand of the resolution of $R$ over $Q$ in homological degrees $i$ and $i+1$ reads 
\[
\xymatrix{  \cdots \ar[r] & R(-i-1)^c \ar[r]^\phi & R(-i)^b \ar[r]^{\psi} &  \cdots}\,,
\]
where $\phi$ is a $b\times c$ matrix of linear forms.  The hypothesis is that $c\leq 2$ and  the desired conclusion is that $\phi$ has no linear syzygies.  This is clear when $c = 1$.  

Assume $c =2$. A linear syzygy of $\phi$ is a linear dependence among its two columns. This means that, with $[l_1, \ldots, l_b]$ and $[m_1, \ldots, m_b]$ denoting the columns, there exist distinct linear forms $r,s$  such that 
\[
r[l_1, \ldots, l_b] = s[m_1, \ldots, m_b]
\]
Now, whenever $l_i$ is nonzero, so must $m_i$ (and vice-versa).  Then since $r\neq s$ we have that $r | m_i$ and $s | l_i$ for each $i$.  But this implies that 
\begin{gather*}
[l_1, \ldots,  l_b] = s[l'_1, \ldots, l'_b] \\
[m_1, \ldots, m_b] = r[m'_1,\ldots, m'_b] 
\end{gather*}
with $l'_{i}$ and $m'_{j}$ in $k$. Since $i\ge 2$, the image of $\psi$ is a submodule of a free module, so it follows that $[l'_i]$ and $[m'_i]$ are themselves syzygies of $\psi$, which is absurd. 

\medskip

The initial ideal $J$ of $I$ plays a role in the remainder of the proof. We repeatedly use, without comment, \eqref{eq:semicontinuity}.

\medskip

(1): From Remark~\ref{re:Taylor} one gets the inequalities below:
\[
\beta_{i,i+1}(R) \leq \beta_{i,i+1}(Q/J) = \beta_{i,i+1}(Q/J_{2} ) \leq \beta_{i}(Q/J_{2})\leq \binom{\beta_1(Q/J_{2})} i = \binom ni\,.
\]
The last equality holds as $\beta_{1,2}(Q/J_{2})=\beta_{1,2}(R)$.   This justifies the first claim in (1). 

Suppose that $\beta_{2,3}(R)=\binom n2$; we may assume $n\ge 2$. One then has equalities
\[
\beta_{2,3}(R)=\beta_{2,3}(Q/J) = \binom n2\,.
\]
By \cite{Pe2}*{Theorem~22.12}, as $I$ is generated by quadrics, there is an equality 
\[
\beta_{2,3}(R) = \beta_{2,3}(Q/J)- \beta_{1,3}(Q/J)\,,
\]
and hence $\beta_{1,3}(Q/J)=0$. Given Buchberger's algorithm, it follows that $J$ is generated in degree two, that is to say, $J=J_{2}$.
In particular, $\beta_{1}(Q/J)=n$. Given this and the equality $\beta_{2,3}(Q/J) = \binom n2$ one deduces that the differential in the Taylor resolution of $Q/J$ is minimal  (meaning, it has coefficients in $Q_{\ges 1}$) in degrees $2$ and $3$. Given the description of the differential in the Taylor resolution, it follows that any pair of monomial generators of $J$ has a common factor, and that the l.c.m.\ of any three of the generators is not equal to the l.c.m.\ of a proper subset. A simple computation then yields that the g.c.d.\ of the generators has degree one and thus $J$, and hence also $I$, has codimension one. Since $I$ is generated by quadrics, it must then be of the form $f(g_{1},\dots,g_{n})$, where $f$ is a linear form and $g_{1},\dots,g_{n}$ is a regular sequence of linear forms. This implies that $I$ has a linear resolution of length $n$.

\medskip

(3): Assume $\reg^{R}_{3}(k)=0$. Then $\beta_{2,j}(R)=0$ for $j\ne 3,4$, by Remark~\ref{re:ACI}(1), so one gets the equality below
\[
\beta_{2}(R) = \beta_{2,3}(R) + \beta_{2,4}(R) \le \binom n2 + \binom n2\,.
\]
The inequalities hold by (1) above and  Remark~\ref{re:ACI}(2), which yields $\beta_{2,4}(R) \leq \binom n2$.

\medskip

(4): This is by Remark~\ref{re:Taylor} as $\beta_{2,3}(Q/J)$ and $\beta_{2,4}(Q/J)$ depend only on $J_{\les 3}$.

\medskip

(5): Since $Q/J$ is a flat degeneration of $R$, the Betti table of $R$ is obtained from that of $Q/J$ via consecutive cancellations; see Peeva~\cite{Pe} and \cite{Pe2}*{Theorem~22.12}.  Thus one has

\medskip

\hbox{
\vbox{
\begin{tabular}{c|clclcl}
$\beta(R)$&0&1&2&3&$\cdots$\\
\hline
0&1&--&--&--&$\cdots$\\
1&--&$n$&$c$&$?$&$\cdots$\\
2&--&--&$d$&$?$&$\cdots$ \\
3&$\cdots$&$\cdots$&$\cdots$&$\cdots$&$\cdots$
\end{tabular}
}
\hskip-200pt
\vbox{
\begin{tabular}{c|clclcl}
$\beta(Q/J)$&0&1&2&3& $\cdots$\\
\hline
0&1&--&--&--&$\cdots$\\
1&--&$n$&$c+m$&$?$&$\cdots$\\
2&--&$m$&$d+\nu$ &$?$&$\cdots$\\
3&$\cdots$&$\cdots$&$\cdots$&$\cdots$&$\cdots$
\end{tabular}
}}
\medskip

This justifies the inequality  in (5).
\end{proof}

\begin{lemma}
\label{le:projdim3}
Assume that $R$ is defined by $3$ quadratic relations.  If $\sum_{i\geq 4} \beta_{2,i}^{Q}(R) \leq 2$ then the projective dimension of $R$ is at most $3$.
\end{lemma}

\begin{proof}
As before, we write $\beta_{i,j}$ for $\beta^{Q}_{i,j}(R)$. Proposition \ref{pr:initial-ideal} (1) yields $\beta_{2,3}\leq 3$ and also that when equality holds $\pdim_{Q}R=3$.  When $\beta_{2,3}\leq 2$ the hypothesis yields $\beta_2 \leq 4$, so $\pdim_{Q}R\leq 3$, by the Syzygy Theorem~\cite{BH}*{Theorem 9.5.6}.
\end{proof}

In what follows we prove the following are the only possible Betti tables for Koszul algebras defined by three relations. The first one is the Betti table of a complete intersection of three quadrics; the second is that of the ring $k[x,y,z]/(x^{2},y^{2},xz)$; the third is defined by the ideal of minors of a $3\times 2$ matrix of linear forms, and of rank two; for example, $k[x,y]/(x,y)^{2}$. The last is the Betti table of an ideal with linear resolution; for example, $k[x,y,z]/(x^{2},xy,xz)$.

\begin{figure}[ht]
\medskip

{\small
\hskip10pt
\begin{tabular}{c|clclclc}
&0&1&2&3\\
\hline
0&1&--&--&--\\
1&--&3&--&--\\
2&--&--&3&--\\
3&--&--&--&1
\end{tabular}
\hskip20pt
\begin{tabular}{c|clclcl}
&0&1&2&3\\
\hline
0&1&--&--&--\\
1&--&3&1&--\\
2&--&--&2&1 
\end{tabular}
\hskip20pt
\begin{tabular}{c|clclcl}
&0&1&2\\
\hline
0&1&--&--\\
1&--&3&2  
\end{tabular}
\hskip20pt
\begin{tabular}{c|clclcl}
&0&1&2&3\\
\hline
0&1&--&--&--\\
1&--&3&3&1 
\end{tabular}
}
\medskip 
\caption{The Betti tables for Koszul algebras defined by three relations}
\end{figure}

\begin{remark}
\label{re:DAli}
Assume $\edim R=3$. It follows from  D'Al\`{i}'s classification~ ~\cite{Ali}*{Theorem 3.1} of quadratic algebras  that the ones that are \emph{not} Koszul have Betti table 
\[
\begin{tabular}{c|clclcl}
&0&1&2&3\\
\hline
0&1&--&--&--\\
1&--&3&--&-- \\ 
2&--&--&4&2  \\
\end{tabular}.
\]
This remark is used in proving (3)$\Rightarrow$(1) and (4)$\Rightarrow$(1) in the result below.
\end{remark}

\begin{theorem}
When $R$ is generated by $3$ quadrics, the following conditions are equivalent.
\begin{enumerate}[\quad\rm(1)]
\item $R$ is Koszul;
\item $\reg^{R}_{n+1}(k)=0$ for $n=\edim R - \dim R$;
\item The Betti table for $R$ over $Q$ is one of those listed in Figure 1.
\item  $H_{*}(K^{R})$ is generated, as a $k$-algebra, by its linear strand. 
\end{enumerate}
\end{theorem}

\begin{proof}
We write $\beta_{i,j}$ for $\beta^{Q}_{i,j}(R)$. We remark that if $R$ is a complete intersection there is little to prove, so we assume that it is not. 

\medskip

(1)$\Rightarrow$(2): This is a tautology.

\medskip

(2)$\Rightarrow$(3): When (2) holds, Remark \ref{re:ACI} and Corollary \ref{Co:diagonalCI} show that $\beta_{2,4} \leq 2$ and $\beta_{2,j} = 0$ for $j>4$.  Now Lemma \ref{le:projdim3} guarantees that the projective dimension of $R$ is at most $3$.  The proof of Proposition~\ref{pr:initial-ideal}(1) shows that if $\beta_{2,3} = 3$ then the Betti table is the last one in the list above.  If $\beta_{2,3}\leq 2$ then Proposition \ref{pr:initial-ideal} (2) yields $\beta_{3,4} = 0$. Finally $\beta_{3,6} = 0$ by Corollary \ref{Co:diagonalCI}, as $R$ is not a complete intersection.  Hence $\beta_3 = \beta_{3,5}$.   The inequalities
\[
\beta_{2,3} \le 2\,\quad\text{and}\quad  \beta_{2,4} \le 2
\]
allow  few possibilities for Betti tables.  A computation using Boij-S\"oderberg theory \cite{BoijSoed} confirms that the only options are the middle two Betti tables in Figure (1).

\medskip

(3)$\Rightarrow$(1):  Evidently, if (3) holds $R$ is quadratic and $\pdim_{Q}R\le 3$. Thus, one can pass to a quotient of $R$ by a regular sequence of linear forms and assume that $\edim R \le 3$.  It then follows from Remark~\ref{re:DAli} that $R$ must be Koszul.

\medskip

(1)$\Rightarrow$(4): We have already verified that (1)$\Rightarrow$(3); the desired implication thus follows from  an inspection of the Betti tables in Figure 1, and Remark~\ref{re:thick-diagonal}.

\medskip

(4)$\Rightarrow$(1): Extending the field $k$, we can assume it is algebraically closed. We first prove that the projective dimension of $R$ must be at most $3$. 

Recall from the proof of Proposition~\ref{pr:initial-ideal}(1) that if $\beta_{2,3} = 3$ then $R$ has a linear resolution of length three.  Since $H_{1,2}(R)$ generates the diagonal by assumption, we know that $\beta_{2,4}\leq 3$.  Assume $\pdim_{Q}R\ge 4$. Then the Syzygy Theorem~\cite{BH}*{Theorem 9.5.6} implies $\beta_2 = 5$ so that $\beta_{2,3} = 2$ and $\beta_{2,4} = 3$.  Let $J$ be the initial ideal of $I$ with respect to some term order. By Proposition \ref{pr:initial-ideal} (5), the ideal $J$ has one cubic generator so the Betti table of $Q/J$ is 
\[
\begin{tabular}{c|clclcl}
&0&1&2&3\\
\hline
0&1&--&--&--\\
1&--&3&3&1 \\ 
2&--&1&$\gamma$ &?  \\
3&--& ? & ?  & ? \\
\end{tabular}.
\]
Since $\beta_{3,4}(R) = 0$, by Proposition \ref{pr:initial-ideal} (2), we get that $\gamma \geq 4$, so the $1$ in this table can cancel.  This contradictions \ref{pr:initial-ideal}(4). This completes the proof that $\pdim_{Q}R\le 3$.

Given this, we can reduce to the case $\edim R = 3$, as in the proof of (3)$\Rightarrow$(1).  Now again we apply Remark~\ref{re:DAli}: if $R$ is not Koszul, then it is clear from the Betti table in \emph{op.~cit.} that $H_*(K^R)$ is not generated by its linear strand.
\end{proof}

The implication (1)$\Rightarrow$(3) in the theorem above settles Question~\ref{qu:conca} when $g\le 3$. It is worth noting that the proof of that implication does not use Remark~\ref{re:DAli}.

\begin{corollary}
\label{co:conca}
\pushQED{\qed}
If $R$ is Koszul and $\beta_1^Q(R)\leq 3$, then
\[
\beta_i^Q(R) \leq \binom{\beta_1^Q(R)}i\quad\text{for each $i\ge 0$.} \hbox{\hfill \qed}
\]
\end{corollary}

To wrap up, we note that the results in this work, combined with those available in the literature, settle Question~\ref{qu:conca} also for algebras of embedding dimension at most three; that is to say, for algebras with at most three generators. 

\begin{remark}
\label{re:conca}
Assume $\rank_{k}(R_{1})\leq 3$ so that $\rank_{k}(R_2) \leq 6$.   A theorem of Backelin and Fr\"oberg \cite{BF}*{Theorem 4.8} shows that if $\rank_{k}(R_2) \leq 2$, the ring $R$ is Koszul, and then Conca~\cite{C2000} proves that $R$ is LG-quadratic, with essentially one exception, and observes in \cite{C2009} that that too is LG-quadratic. The inequality in Question~\ref{qu:conca} follows; see Remark~\ref{re:Taylor}. When  $\rank_{k}(R_{2})\geq 3$ and $R$ is Koszul, Corollary~\ref{co:conca} leads to the same conclusion.
\end{remark}

\begin{ack}
It is a pleasure to thank the referee for a careful reading of the manuscript. The second author was supported by a Post-Doctoral fellowship from CNPq-Brazil. He would like to thank CNPq for the support, and the university of Utah for the hospitality during the time of the preparation this work. The third author was partly supported by  NSF grant DMS-1503044.
\end{ack}


\begin{bibdiv}
  \begin{biblist}

\bib{Ao}{article}{
   author={Aoyama, Y{\^o}ichi},
   title={On the depth and the projective dimension of the canonical module},
   journal={Japan. J. Math. (N.S.)},
   volume={6},
   date={1980},
   number={1},
   pages={61--66},
}
	
\bib{Av1}{article}{
   author={Avramov, Luchezar L.},
   title={Infinite free resolutions [MR1648664]},
   conference={
      title={Six lectures on commutative algebra}, },
   book={
      series={Mod. Birkh\"auser Class.},
      publisher={Birkh\"auser Verlag, Basel}, },
   date={2010},
   pages={1--118},
}

\bib{ACI1}{article}{
   author={Avramov, Luchezar L.},
   author={Conca, Aldo},
   author={Iyengar, Srikanth B.},
   title={Free resolutions over commutative Koszul algebras},
   journal={Math. Res. Lett.},
   volume={17},
   date={2010},
   number={2},
   pages={197--210},
}

\bib{ACI}{article}{
   author={Avramov, Luchezar L.},
   author={Conca, Aldo},
   author={Iyengar, Srikanth B.},
   title={Subadditivity of syzygies of Koszul algebras},
   journal={Math. Ann.},
   volume={361},
   date={2015},
   number={1-2},
   pages={511--534},
}
		

\bib{BThesis}{article}{
	author = {Backelin, J.},
	title ={Relations between rates of growth of homologies},
	journal={Research Reports in Mathematics},
	number={25, Department of Mathematics},
	pages={Stockholm University},
	date={1988},
}

\bib{BF}{article}{
   author={Backelin, J{\"o}rgen},
   author={Fr{\"o}berg, Ralf},
   title={Koszul algebras, Veronese subrings and rings with linear
   resolutions},
   journal={Rev. Roumaine Math. Pures Appl.},
   volume={30},
   date={1985},
   number={2},
   pages={85--97},
}


\bib{BDGMS}{article}{
   author={Boocher, Adam},
   author={D'Al{\`{\i}}, Alessio},
   author={Grifo, Elo{\'{\i}}sa},
   author={Monta{\~n}o, Jonathan},
   author={Sammartano, Alessio},
   title={Edge ideals and DG algebra resolutions},
   journal={Matematiche (Catania)},
   volume={70},
   date={2015},
   number={1},
   pages={215--238},
}	

\bib{BoijSoed}{article}{
   author={Eisenbud, David},
   author={Schreyer, Frank-Olaf},
   title={Betti numbers of graded modules and cohomology of vector bundles},
   journal={J. Amer. Math. Soc.},
   volume={22},
   date={2009},
   number={3},
   pages={859--888},
}

\bib{Br}{article}{
   author={Bruns, Winfried},
   title={``Jede'' endliche freie Aufl\"osung ist freie Aufl\"osung eines
   von drei Elementen erzeugten Ideals},
   journal={J. Algebra},
   volume={39},
   date={1976},
   number={2},
   pages={429--439},
}

\bib{Br2}{article}{
   author={Bruns, Winfried},
   title={On the Koszul algebra of a local ring},
   journal={Illinois J. Math.},
   volume={37},
   date={1993},
   number={2},
   pages={278--283},
}


\bib{BH}{book}{
   author={Bruns, Winfried},
   author={Herzog, J{\"u}rgen},
   title={Cohen-Macaulay rings},
   series={Cambridge Studies in Advanced Mathematics},
   volume={39},
   publisher={Cambridge University Press, Cambridge},
   date={1993},
   pages={xii+403},
}

\bib{C2000}{article}{
   author={Conca, Aldo},
   title={Gr\"obner bases for spaces of quadrics of low codimension},
   journal={Adv. in Appl. Math.},
   volume={24},
   date={2000},
   number={2},
   pages={111--124},
   issn={0196-8858},
}

\bib{C2009}{article}{
   author={Conca, Aldo},
   title={Koszul algebras and Gr\"obner bases of quadrics},
   date={2009},
   status={preprint},
   eprint={https://arxiv.org/abs/0903.2397},
}

\bib{Co}{article}{
   author={Conca, Aldo},
   title={Koszul algebras and their syzygies},
   conference={
      title={Combinatorial algebraic geometry},
   },
   book={
      series={Lecture Notes in Math.},
      volume={2108},
      publisher={Springer, Cham},
   },
   date={2014},
   pages={1--31},
}


\bib{Ali}{article}{
   author={D'Al\`i, Alessio},
   title={The Koszul property for spaces of quadrics of codimension three},
   date={2016},
   status={preprint},
   eprint={https://arxiv.org/abs/1605.09145}
}

\bib{HN}{article}{
   author={Hassanzadeh, Seyed Hamid},
   author={Na{\'e}liton, Jose},
   title={Residual intersections and the annihilator of Koszul homologies},
   journal={Algebra Number Theory},
   volume={10},
   date={2016},
   number={4},
   pages={737--770},
}


\bib{JL}{article}{
   author={Jorgensen, David A.},
   author={Leuschke, Graham J.},
   title={On the growth of the Betti sequence of the canonical module},
   journal={Math. Z.},
   volume={256},
   date={2007},
   number={3},
   pages={647--659},
}

\bib{MR2434477}{article}{
   author={Jorgensen, David A.},
   author={Leuschke, Graham J.},
   title={Erratum: ``On the growth of the Betti sequence of the canonical
   module'' [Math. Z. {\bf 256} (2007), no. 3, 647--659; MR 2299575]},
   journal={Math. Z.},
   volume={260},
   date={2008},
   number={3},
   pages={713--715},
}

\bib{Ke}{article}{
   author={Kempf, George R.},
   title={Some wonderful rings in algebraic geometry},
   journal={J. Algebra},
   volume={134},
   date={1990},
   number={1},
   pages={222--224},
}

\bib{Ku}{article}
{
   author={Kunz, Ernst},
   title={Almost complete intersections are not Gorenstein rings},
   journal={J. Algebra},
   volume={28},
   date={1974},
   pages={111--115},
   review={\MR{0330158}},
}


\bib{MS}{article}{
   author={McCullough, Jason},
   author={Seceleanu, Alexandra},
   title={Bounding projective dimension},
   conference={
      title={Commutative algebra},
   },
   book={
      publisher={Springer, New York},
   },
   date={2013},
   pages={551--576},
}

\bib{Pe}{article}{  
author={Peeva, Irena},
   title={Consecutive cancellations in Betti numbers},
   journal={Proc. Amer. Math. Soc.},
   volume={132},
   date={2004},
   number={12},
   pages={3503--3507},
}

\bib{Pe2}{book}{
   author={Peeva, Irena},
   title={Graded syzygies},
   series={Algebra and Applications},
   volume={14},
   publisher={Springer-Verlag London, Ltd., London},
   date={2011},
   pages={xii+302},
}

\bib{Roos}{article}{
   author={Roos, Jan-Erik},
   title={Commutative non-Koszul algebras having a linear resolution of
   arbitrarily high order. Applications to torsion in loop space homology},
   language={English, with English and French summaries},
   journal={C. R. Acad. Sci. Paris S\'er. I Math.},
   volume={316},
   date={1993},
   number={11},
   pages={1123--1128},
   issn={0764-4442},
   review={\MR{1221635}},
}
\end{biblist}
\end{bibdiv}
\end{document}